\documentclass[12pt, reqno]{article}

\usepackage[usenames]{color}
\usepackage{amssymb}
\usepackage{amsmath}
\usepackage{amsthm}
\usepackage{amsfonts}
\usepackage{amscd}
\usepackage{graphicx}
\usepackage{hyperref}
\usepackage{mathtools}
\usepackage{fullpage}
\usepackage{comment}

\theoremstyle{plain}
\newtheorem{theorem}{Theorem}

\newtheorem{lemma}[theorem]{Lemma}

\theoremstyle{definition}

\theoremstyle{remark}
\newtheorem{remark}[theorem]{Remark}

\def\N{\mathbb{N}}

\def\Z{\mathbb{Z}}

\def\F{\mathbb{F}}

\author{Luis H. Gallardo \\
Univ. Brest, \\
UMR CNRS 6205, \\
Laboratoire de Math\'ematiques de Bretagne Atlantique,\\
6, Av. Le Gorgeu, C.S. 93837, Cedex 3, F-29238 Brest, \\
France \\
{\tt Luis.Gallardo@univ-brest.fr}\\
 AMS 2010: Primary 11T55, Secondary 11T06,\\
  05A15, 11B13.
Running Head: Binary Polynomials}

\title{Splitting sums of binary polynomials}

\begin{document}

\maketitle

\begin{abstract}
We study an analogue of a classical arithmetic problem over the ring of polynomials. We prove that $m = 5$ is the minimal number such that the sums of any two distinct polynomials in a set of $m$ polynomials over $\F_2[x]$ cannot all be of the form $x^k(x+1)^{\ell}$.
\end{abstract}

\newpage

\section{Introduction}

It is easy to see that there exist distinct integers $a, b$ such that $a + b$ is a power of $2$, e.g., $a = 3$, $b = 5$. Moreover, there exist distinct integers $a, b, c$ such that all pairwise sums $a + b$, $a + c$, and $b + c$ are powers of $2$, e.g., $a = -1$, $b = 3$, $c = 5$. However, M. S. Smith \cite[sequence A352178]{oeis} proved that there does not exist a set of four integers $\{a, b, c, d\}$ such that $a + b$, $b + c$, $c + d$, and $d + a$ are all powers of $2$. Consequently, no four distinct integers $a, b, c, d$ exist such that all six pairwise sums among them are powers of $2$.

In this paper, we investigate a polynomial analogue of the above problem. We replace an integer $n$ with a polynomial $A(x)$ having coefficients in $\{0,1\}$ and operate in the field $\F_2 = \{0,1\}$, where the rule $1+1 = 0$ replaces the usual $1+1 = 2$. We call such polynomials \emph{binary polynomials}, and the set of them forms the ring $\F_2[x]$. Some basic computations in this ring are shown after Remark~\ref{computs}.
 
The ring $\F_2[x]$ serves as a natural analogue of the integers $\Z$, and certain arithmetic problems become more tractable in this context. We arbitrarily consider the polynomial $x^a(x+1)^b \in \F_2[x]$ as a polynomial analogue of $2^{a+b} \in \Z$. This analogy is motivated by the fact that $2$ is the smallest prime number in $\Z$, while $x$ and $x+1$ are the irreducible polynomials of smallest degree in $\F_2[x]$. Moreover, we cannot associate $2$ with $x$ alone-ignoring $x+1$-since the rings $\F_2[x]$ and $\F_2[x+1]$ are essentially the same. 
In this paper, we focus on the following problem in the ring $\F_2[x]$, inspired by the integer case above.

Given a positive integer $m$, let $S_m$ be a set of $m$ binary polynomials such that the sum of each pair of distinct elements in $S_m$ splits over $\F_2$. That is, the sum is of the form $x^k(x+1)^\ell$ for some non-negative integers $k, \ell$, not both zero.

For $m = 2$, the answer is straightforward: choose any $a \in \F_2[x]$ and let $b = a + x^k(x+1)^\ell$. Then $S_2 = \{a, b\}$ satisfies the conditions for any $(k, \ell) \neq (0, 0)$. The goal of this paper is to investigate what happens for $m > 2$.

Our main results are as follows:

\begin{theorem}
\label{Nsedaka}
Assume that $a, b, c \in \F_2[x]$ satisfy
\begin{equation}
\label{tres}
a + b, \quad a + c, \quad \text{and} \quad b + c \quad \text{split over } \F_2.
\end{equation}
That is, 
\[
a + b = x^{a_1}(x+1)^{b_1}, \quad 
a + c = x^{a_2}(x+1)^{b_2}, \quad 
b + c = x^{a_3}(x+1)^{b_3},
\]
with $(a_j, b_j) \neq (0, 0)$, and $a_1 \leq a_2 \leq a_3$.

Then (up to switching $x$ and $x+1$), the following holds:
 
\begin{equation}
\label{Rtres1}
b = a + x^{a_1}(x+1)^{b_2 + 2^s}, \quad 
c = a + x^{a_1+2^s}(x+1)^{b_2}
\end{equation}
for some non-negative integer $s$;

\end{theorem}
 
\begin{theorem}
\label{Nsedaka1}
\begin{itemize}
\item[{\rm (i)}]
Let $a, b, c, d \in \F_2[x]$ be such that all pairwise sums
\[
a + b, \quad a + c, \quad a + d, \quad b + c, \quad b + d, \quad c + d
\]
split over $\F_2$. That is, each sum equals $x^{a_j}(x+1)^{b_j}$ with $(a_j, b_j) \neq (0, 0)$ for $j = 1, \dots, 6$.

Then (up to switching $x$ and $x+1$), either:
\begin{equation}
\label{Rtres21}
b = a + (x+1)^{2^t} T(1,3), \quad 
c = a + (x+1)^{2^{t-1}} T(1,3), \quad 
d = a + T(1,3)
\end{equation}
where $T(1,3) = x^{a_1}(x+1)^{b_3}$ and $t$ is a non-negative integer, or
\begin{equation}
\label{Rtres22}
b = a + (x^{2^{t_1}} + 1) T(1,3), \quad 
c = a + (x^{2^{t_1}} + x^{2^{t_1 - 1}}) T(1,3), \quad 
d = a + x^{2^{t_1}} T(1,3)
\end{equation}
for some non-negative integer $t_1$.

\item[{\rm (ii)}]
For any $m > 4$, there do not exist distinct binary polynomials $k_1, \dots, k_m$ such that $k_i + k_j$ splits over $\F_2[x]$ for all $i \ne j$.
\end{itemize}
\end{theorem}

\begin{remark}
\label{computs}
Before solving the case $m = 4$, we checked (by computer) that there are no solutions for $m = 5$ when the polynomials $a, b, c, d, e \in \F_2[x]$ have degrees at most $9$.
\end{remark}

Theorem~\ref{Nsedaka} is proved in Section~\ref{dos}, and Theorem~\ref{Nsedaka1} in Section~\ref{tres}. The tools used in the proofs are introduced in Section~\ref{sec:tools}.

We define $\sigma(A)$ to be the sum of all divisors of $A \in \F_2[x]$, including $1$ and $A$ itself. For instance:
\[
\sigma(0) = 0, \quad 
\sigma(1) = 1, \quad 
\sigma(x) = x + 1, \quad 
\sigma(x^2) = x^2 + x + 1, \quad 
\sigma(x^2 + x) = x^2 + x,
\]
\[
\sigma(x^2 + x + 1) = 1 + x^2 + x + 1 = x^2 + x.
\]
Here is why: the divisors of $x^2$ are $1$, $x$, and $x^2$, summing to $1 + x + x^2$. Since $x^2 + x + 1$ is irreducible over $\F_2$, its only divisors are $1$ and itself, so the sum is $1 + (x^2 + x + 1) = x^2 + x$. Similarly, if $P$ is irreducible, then
 $$
 \sigma(P^k) = 1 + P + \cdots + P^k.
 $$ 
 Also, if $A$ and $B$ are coprime polynomials in  $\F_2[x]$, then 
 $$
 \sigma(A B) = \sigma(A) \cdot \sigma(B).
 $$ Thus,
\begin{equation}
\label{perf1}
\sigma(x(x+1)) = \sigma(x) \cdot \sigma(x+1) = (x+1)x = x(x+1).
\end{equation}

A polynomial $A \in \F_2[x]$ is called \emph{Mersenne} if $A = x^a(x+1)^b + 1$ for some $a, b \in \N$. This is a polynomial analogue of the Mersenne number $2^{a+b} - 1$. If $A$ is irreducible, we call it a \emph{Mersenne prime}.

\begin{remark}
\label{Mersennes}
A  Mersenne polynomial $1+x^a(x+1)^b$ is prime if and only if  the trinomial $x^{a+b}+x^b+1$  modulo $2$ is irreducible (see  \cite[Theorem 1.3]{Gall-Rahav8}, see also \cite{BrentLarvalaZ}). Moreover, they are also related to perfect polynomials (see \cite{Gall-Rahav8,Gall-Rahav2020}), i.e. to fixed points of the function $\sigma$ over $\F_2[x]$.
\end{remark}

\section{Tools}
\label{sec:tools}

The following lemma is taken from \cite[Lemma 5]{Canaday}.

\begin{lemma}
\label{L5Canaday}
Let \( P, Q \in \F_2[x] \) be such that \( P \) is irreducible, and let \( n, m \) be non-negative integers such that
\[
1 + P + \cdots + P^{2n} = Q^m.
\]
Then \( m \in \{0,1\} \).
\end{lemma}

The next lemma solves a simple exponential equation over \( \F_2[x] \).

\begin{lemma}
\label{ABC}
The only non-negative integers \( A, B, C \) satisfying
\begin{equation}
\label{xabc}
(x+1)^A + x^B = x^C
\end{equation}
in \( \F_2[x] \)
are either \( A = 2^s, B = 0, C = 2^s \) or \( A = 2^s, B = 2^s, C = 0 \), for some non-negative integer \( s \).
\end{lemma}

\begin{proof}
If \( A = 0 \), then \( 1 + x^B = x^C \), implying \( B = C \) and hence \( 1 + x^B = x^B \), which is a contradiction. So \( A \geq 1 \). Suppose \( A \) is not a power of \( 2 \). Then clearly \( B = 0 \) and \( C = 0 \) is not possible. If either \( B \) or \( C \) is zero, say \( B = 0 \) and \( C \geq 1 \), then
\begin{equation}
\label{starA1}
(x+1)^A = x^C + 1.
\end{equation}
Dividing both sides by \( x+1 \), we obtain
\begin{equation}
\label{starA}
(x+1)^{A-1} = \frac{x^C+1}{x+1} = \sigma(x^{C-1}).
\end{equation}
Taking degrees in \eqref{starA} gives \( C = A \), so we have
\begin{equation}
\label{starB}
(x+1)^{A-1} = \sigma(x^{A-1}).
\end{equation}
If \( A - 1 \) is even, Lemma \ref{L5Canaday} leads to the contradiction \( A - 1 = 1 \). Thus \( A \) must be even. Write \( A = 2^s u \) with \( s \geq 1 \) and \( u > 1 \) odd. Then \eqref{starA1} becomes
\begin{equation}
\label{starC}
((x+1)^u)^{2^s} = (x^u + 1)^{2^s},
\end{equation}
which implies
\begin{equation}
\label{starC1}
(x+1)^u = x^u + 1.
\end{equation}
Dividing by \( x+1 \) again, we get
\begin{equation}
\label{starC2}
(x+1)^{u-1} = \sigma(x^{u-1}),
\end{equation}
with \( u - 1 \) even. But this contradicts Lemma \ref{L5Canaday}.

Hence, \( A = 2^s \) for some \( s \geq 0 \). Substituting into \eqref{xabc}, we get
\begin{equation}
\label{starF}
(x+1)^{2^s} + 1 = x^B + x^C + 1,
\end{equation}
which simplifies to
\begin{equation}
\label{starG}
x^{2^s} = x^B + x^C + 1.
\end{equation}
Therefore, \( B = 0, C = 2^s \) or \( C = 0, B = 2^s \), as desired.
\end{proof}

A \emph{Sidon} sequence or set is a sequence \( S = \{s_0, s_1, s_2, \ldots\} \) of natural numbers in which all pairwise sums \( s_i + s_j \) with \( i \leq j \) are distinct.

The following elementary lemma, whose proof is left to the reader, is useful. In particular, it implies the well-known result that the set of powers of 2 forms a Sidon subsequence in \( \Z \) (see \cite[Section C9]{Guy}).

\begin{lemma}
\label{powers2}
The only non-negative integers \( A, B, C \) satisfying
\begin{equation}
\label{xxabc}
2^A = 2^B + 2^C
\end{equation}
are those with \( A = B + 1 = C + 1 \).
\end{lemma}

\begin{lemma}
\label{reduction}
Let \( A, B, C, D, E, F \) be non-negative integers such that
\begin{equation}
\label{xabcdef}
x^A(x+1)^B + x^C(x+1)^D = x^E(x+1)^F
\end{equation}
holds in \( \F_2[x] \). Then, after reordering and possibly swapping terms, we may assume \( A \leq C \leq E \). Moreover, if \( C = A \), we may also assume \( B \geq D \).
\end{lemma}

\begin{proof}
We can write \eqref{xabcdef} as
\begin{equation}
\label{letters}
x^A(x+1)^B + x^C(x+1)^D + x^E(x+1)^F = 0.
\end{equation}
By relabeling, we may assume \( A \leq C \leq E \). If \( C = A \), then
\eqref{xabcdef} becomes
\begin{equation}
\label{letters1}
(x+1)^B + (x+1)^D = x^{E-A}(x+1)^F,
\end{equation}
so that we may further assume \( B \geq D \).
\end{proof}

The following lemma plays a key role in our proofs. Moreover, it shows that the set of split polynomials in \( \F_2[x] \) is far from being a Sidon set; that is, the sums \( a + b \) with \( a, b \in S \) and \( a \neq b \) are not all distinct.

\begin{lemma}
\label{ACEsept}
Let \( A, B, C, D, E, F \) be non-negative integers such that
\begin{equation}
\label{xabcdef1}
x^A(x+1)^B + x^C(x+1)^D = x^E(x+1)^F
\end{equation}
in \( \F_2[x] \), with \( A \leq C \leq E \). If \( C = A \), then \( B \geq D \).

Then \( \min(B, D, F) = D \). Moreover,
   
     \( B \neq D \), \( C = A \), \( F = D \), and:
    
         \( B - D = E - A = 2^s \) for some non-negative integer \( s \).
       
\end{lemma}

\begin{proof}
Assume that \eqref{xabcdef1} holds with \( A \leq C \leq E \), and if \( C = A \), then \( B \geq D \). Let us write
\[
P := x^A(x+1)^B, \quad Q := x^C(x+1)^D, \quad R := x^E(x+1)^F.
\]

We proceed by dividing the equation
\[
P + Q = R
\]
by the appropriate power of \( x+1 \) to reduce the minimal exponent among \( B, D, F \) to zero.

\textbf{Step 1: Normalize by minimal power of \( x+1 \).}

Let \( m = \min(B, D, F) \). Dividing both sides of the equation by \( (x+1)^m \), we obtain
\begin{equation}
\label{eq:reduced}
x^A(x+1)^{B - m} + x^C(x+1)^{D - m} = x^E(x+1)^{F - m}.
\end{equation}
We claim that \( m = D \). Suppose not.

\textbf{Case 1: \( m = B < D \).} Then in \eqref{eq:reduced}, the left-hand side becomes
\[
x^A + x^C(x+1)^{D - B}.
\]
Since \( D - B \geq 1 \), this is a sum of two split polynomials with distinct powers of \( x+1 \). But the right-hand side \( x^E(x+1)^{F - B} \) is also a split polynomial. This contradicts Lemma \ref{ABC}, which classifies when such a sum equals a single split polynomial. Hence, this case cannot occur.

\textbf{Case 2: \( m = F < D \).} Then the right-hand side becomes \( x^E \), a monomial. But then the left-hand side is a sum of two split polynomials, which again cannot equal a monomial unless one of them is zero, which is excluded by assumption. Thus, this case also leads to a contradiction.

Therefore, we must have \( m = D \), as claimed.

\textbf{Step 2: Analyze according to whether \( B = D \) or \( B \neq D \).}

We now consider \eqref{xabcdef1}
with \( D \) minimal among \( B, D, F \). Divide both sides by \( (x+1)^D \), yielding
\begin{equation}
\label{eq:step2}
x^A(x+1)^{B - D} + x^C = x^E(x+1)^{F - D}.
\end{equation}

\textbf{Subcase 1: \( B = D \).}

Then \eqref{eq:step2} becomes
\begin{equation}
\label{eq:step2A}
x^A + x^C = x^E(x+1)^{F - D}
\end{equation}
so that $A < C \leq E$. Taking degrees into  \eqref{eq:step2A} we obtain that
\[
C = E + (F - D) \geq E.
\]
Thus, $C=E$, and $F=D$. Hence, \eqref{eq:step2A} gives the contradiction 
\[
x^A+x^C = x^C.
\]
In other words, the case $B=D$ does not happen.

\textbf{Subcase 2: \( B > D \).}

Then \eqref{eq:step2} becomes
\[
x^A(x+1)^{B - D} + x^C = x^{E}(x+1)^{F - D}.
\]

 Now two cases arise:

\begin{itemize}
    \item If \( C = A \), then \eqref{eq:step2} becomes
    \begin{equation}
    \label{eq:beefe}
    (x+1)^{B - D} + 1 = x^{E - A}(x+1)^{F - D}.
    \end{equation}
    
    We claim that $B - D > F -D \geq 0$. Otherwise, \eqref{eq:beefe} gives the contradiction
  \[
    (x+1)^{B - D} \mid 1.
    \]
    
    Thus, \eqref{eq:beefe} implies that $(x+1)^{F - D} \mid 1$, so that $F = D$.
    Therefore,  \eqref{eq:beefe} becomes
    \begin{equation}
    \label{eq:paraL6}
     1 = x^{E - A} + (x+1)^{B - D}.
    \end{equation}

   Apply Lemma \ref{ABC} to  \eqref{eq:paraL6}, which implies  \( B - D = E - A = 2^s \) for some \( s \geq 0 \).

   \item If \( C \neq A \), then we can write \eqref{eq:step2} as follows:
   
   \begin{equation}
    \label{eq:aete}
     x^A \left( (x+1)^{B-D} + x ^{C-A} \right) = x^{E} (x+1)^{F - D}.
    \end{equation}
    
   Comparing the exponent of $x$ in both sides of \eqref{eq:aete} we reach the contradiction $A = E$. In other words, the case $C \neq A$ does not happen.
 
\end{itemize}
This completes the analysis of all cases, and hence the proof.
\end{proof}

\section{Proof of Theorem \ref{Nsedaka}}
\label{dos}

By adding $a + b$ to $a + c$, we obtain the identity
\begin{equation}
\label{dos1}
x^{a_1}(x+1)^{b_1} + x^{a_2}(x+1)^{b_2} = x^{a_3}(x+1)^{b_3}.
\end{equation}

By Lemma \ref{reduction}, we may assume that if $a_2 = a_1$, then $b_1 \geq b_2$.  
Apply Lemma \ref{ACEsept} with $A = a_1$, $B = b_1$, $C = a_2$, $D = b_2$, $E = a_3$, and $F = b_3$. We obtain that $\min(b_1, b_2, b_3) = b_2$. Moreover: $a_2=a_1$,  $b_3=b_2$, $a_3 =a_1 + 2^s$, and  $b_1 =b_2 + 2^s$ for some integer $s \geq 0$. 
This completes the proof.
\hfill $\Box$

\section{Proof of Theorem \ref{Nsedaka1}}
\label{tres}

To prove (i), we proceed as in the proof of Theorem \ref{Nsedaka}. The proof is divided into three parts: Part (a), Part (b), and Part (c). In Part (a), we list all sixteen cases to consider. In Part (b), we give a detailed proof of the two cases that hold. In Part (c), we prove that two of the remaining fourteen cases do not occur. The analysis of the other twelve cases (which also do not occur) is similar and therefore omitted. The computational verification of Part (ii) completes the proof of the theorem.

\textbf{Part (a).} The list $L$ of the $16$ cases to consider is as follows:
\begin{align}
\label{sixteen1}
L &= \{[1A,2A,3A,4A], [1B,2A,3A,4A], [1A,2B,3A,4A], [1A,2A,3B,4A]\} \cup \\
\label{sixteen2}
&\quad \{[1A,2A,3A,4B], [1B,2B,3A,4A], [1B,2A,3B,4A], [1B,2A,3A,4B]\} \cup \\
\label{sixteen3}
&\quad \{[1A,2B,3B,4A], [1A,2B,3A,4B], [1A,2A,3B,4B], [1A,2B,3B,4B]\} \cup \\
\label{sixteen4}
&\quad \{[1B,2A,3B,4B], [1B,2B,3A,4B], [1B,2B,3B,4A], [1B,2B,3B,4B]\},
\end{align}
where
\begin{itemize}
\item
1A: $a_2 = a_1$, $a_4 = a_1 + 2^s$, $b_1 = b_2 + 2^s$, $b_4 = b_2$.\\
2A: $a_3 = a_1$, $a_5 = a_1 + 2^t$, $b_1 = b_3 + 2^t$, $b_5 = b_3$.\\
3A: $a_5 = a_4$, $a_6 = a_4 + 2^u$, $b_4 = b_5 + 2^u$, $b_6 = b_5$.\\
4A: $a_3 = a_2$, $a_6 = a_2 + 2^v$, $b_2 = b_3 + 2^v$, $b_6 = b_3$.
\item
1B: $a_2 = a_1 + 2^{s_1}$, $a_4 = a_1$, $b_1 = b_2 + 2^{s_1}$, $b_4 = b_2$.\\
2B: $a_3 = a_1 + 2^{t_1}$, $a_5 = a_1$, $b_1 = b_3 + 2^{t_1}$, $b_5 = b_3$.\\
3B: $a_5 = a_4 + 2^{u_1}$, $a_6 = a_4$, $b_4 = b_5 + 2^{u_1}$, $b_6 = b_5$.\\
4B: $a_3 = a_2 + 2^{v_1}$, $a_6 = a_2$, $b_2 = b_3 + 2^{v_1}$, $b_6 = b_3$.
\end{itemize}

\textbf{Part (b).} Consider the case $[1A,2A,3B,4A]$, which leads to \eqref{Rtres21}. From $2A$ and $1A$, we obtain $a_5 - a_4 = 2^t - 2^s$, while $3B$ implies $a_5 - a_4 = 2^{u_1}$. Thus, $2^t = 2^s + 2^{u_1}$. By \eqref{xabc}, we obtain $s = t-1$ and $u_1 = t - 1$. From $1A$ and $2A$, we have $b_2 + 2^s = b_3 + 2^t$, hence $b_2 - b_3 = 2^t - 2^s$. Now $4A$ implies $b_2 - b_3 = 2^v$, so that $2^t = 2^s + 2^v$. By \eqref{xabc}, we get $s = t - 1$ and $v = t - 1$. Putting everything together, we obtain the result.

Similarly, the other valid case $[1B,2B,3A,4B]$ leads to \eqref{Rtres22}.

\textbf{Part (c).} Consider the case $[1A,2A,3A,4A]$. From $1A$ and $2A$, we get $b_1 = b_2 + 2^s = b_3 + 2^t$, so that $b_2 - b_3 = 2^t - 2^s$. Meanwhile, $4A$ implies $b_2 - b_3 = 2^v$, and therefore $2^t = 2^s + 2^v$. By \eqref{xabc}, this yields
\begin{equation}
s = t - 1 \quad \text{and} \quad v = t - 1.
\label{star1}
\end{equation}
From $1A$, we have $a_4 - a_2 = 2^s$. From $3A$ and $4A$, we know $a_6 = a_4 + 2^u$ and $a_6 = a_2 + 2^v$, so that $a_4 - a_2 = 2^v - 2^u$, and thus \eqref{xabc} implies
\begin{equation}
u = v - 1 \quad \text{and} \quad s = v - 1.
\label{star2}
\end{equation}
Combining \eqref{star1} and \eqref{star2}, we obtain the contradiction $v = s = v - 1$. This proves that this case does not occur.

Now consider the case $[1B,2A,3A,4B]$. From $1B$, we get $a_2 - a_1 = 2^{s_1}$. From $4B$, we get $a_3 - a_2 = 2^{v_1}$. From $2A$, we also know $a_3 = a_1$, so that $a_1 - a_2 = 2^{v_1}$. Therefore, $0 = 2^{s_1} + 2^{v_1}$, a contradiction. This proves that this case also does not occur.

Similar arguments show that the remaining twelve cases do not occur. This proves Part (i) of the theorem.

The proof of Part (ii) follows from a straightforward computation using GP-PARI. We checked computationally (in a few seconds) that for $m = 5$, each of the $64$ possible cases contained at least one of the fourteen cases that do not occur (as shown in the proof of Part (i)).

This completes the proof of the theorem.
\hfill $\Box$

\section{Conclusion}

We translated the problem of distinct sums of powers of $2$ into the setting of distinct sums of certain binary polynomials. While the polynomial case required a more involved analysis, we obtained a complete solution using only elementary theoretical properties of these polynomials, along with some simple computer computations. The problem appears to be difficult—if not impossible—to solve using purely computational methods.

\section{Acknowledgments}
We thank the referee for valuable suggestions that led to a substantial simplification of the statements and proofs of Lemma \ref{ACEsept} and Theorem \ref{Nsedaka}.


\begin{thebibliography}{99}

\bibitem{BrentLarvalaZ} R. P. Brent, S. Larvala, and P. Zimmermann,  
A fast algorithm for testing reducibility of trinomials mod 2 and some new primitive trinomials of degree 3021377,  
\textit{Math. Comput.} \textbf{72} (2003), 1443--1452.

\bibitem{Canaday} E. F. Canaday,  
The sum of the divisors of a polynomial,  
\textit{Duke Math. J.} \textbf{8} (1941), 721--737.

\bibitem{Gall-Rahav8} L. H. Gallardo and O. Rahavandrainy,  
On Mersenne polynomials over $\F_{2}$,  
\textit{Finite Fields Appl.} \textbf{59} (2019), 284--296.

\bibitem{Gall-Rahav2020} L. H. Gallardo and O. Rahavandrainy,  
A polynomial variant of perfect numbers,  
\textit{J. Integer Seq.} \textbf{23} (2020), Article 20.8.6, 9 pp.

\bibitem{Guy} R. K. Guy,  
\textit{Unsolved Problems in Number Theory}, 3rd ed., Springer-Verlag, 2004.

\bibitem{Rudolf} R. Lidl and H. Niederreiter,  
\textit{Finite Fields}, Encyclopedia of Mathematics and its Applications, Vol.~20,  
Cambridge University Press, 1996, xiv+755 pp.

\bibitem{oeis} N. J. A. Sloane et al.,  
\textit{The On-Line Encyclopedia of Integer Sequences}, published electronically at  
\url{https://oeis.org}, 2019.

\end{thebibliography}
\end{document}